\documentclass[12pt,reqno]{article}
\usepackage{hyperref}
\usepackage{amsmath, amsthm, graphicx,amsfonts, amssymb,color}
\usepackage{bm}

\newtheorem{thm}{Theorem}[section]
\newtheorem{cor}[thm]{Corollary}
\newtheorem{lem}[thm]{Lemma}
\newtheorem{prop}[thm]{Proposition}
\theoremstyle{definition}

\theoremstyle{remark}
\newtheorem{rem}[thm]{Remark}
\theoremstyle{example}

\numberwithin{equation}{section}
\usepackage{mathrsfs}

\newcommand{\R}{\mathbb R}

\newcommand{\E}{\mathbb{E}}

\newcommand{\nn}{\nonumber}

\def\e{\mathcal E}

\def\N{\mathbb N}
\def\cL{\mathcal L}

\def\cN{\mathcal N}

\def\P{\mathbb P}
\def\R{\mathbb{R}}

\def\bg{\begin}
\def\be{\bg{equation}}
\def\de{\end{equation}}

\def\hat{\widehat}
\def\bar{\overline}

\title{\bf {Variational principles for the exit time of non-symmetric diffusions}}

\author{
{\bf Lu-Jing Huang}\\
{\small School of Mathematics and Statistics \& Key Laboratory of Analytical Mathematics and Applications} \\
{\small (Ministry of Education) \& Fujian Provincial Key Laboratory of Statistics and Artificial Intelligence}\\
{\small \& Fujian Key Laboratory of Analytical Mathematics and Applications (FJKLAMA) }\\
{\small \& Center for Applied Mathematics of Fujian Province (FJNU), Fujian Normal University}\\
{\bf Kyung-Youn Kim}\\
{\small  Department of Applied Mathematics, National Chung Hsing University }\\
{\bf Yong-Hua Mao}\\
{\small  Laboratory of Mathematics and Complex Systems (Ministry of Education),}\\ {\small School of Mathematical Sciences, Beijing Normal University}
}

\date{}

\begin{document}

 \maketitle


\begin{abstract}

In this paper we develop some new variational principles for the exit time of non-symmetric diffusions from a domain. As applications, we give some comparison theorems and monotonicity law between different diffusions.

\end{abstract}



{\bf Running head} Variational principles for the exit time

\section{Introduction}\label{intro}

The exit time of a diffusion from a domain has been widely used in applied mathematics
when certain problems are considered as underlying stochastic processes that eventually will reach (escape) a specific level (see e.g. \cite{Ca17,RMO14}).
It also plays an important role in probability theory.
{Indeed,} the exit time is the starting point for the study of the ergodicity (see e.g. \cite{Ch05,IW89,Wa05}), and it provides some good estimates of the spectral gap of diffusions (see e.g. \cite{DLM17,LLS11,Mc13,VFNT16}).
There are many studies {that} focus on the exit time, such as the distribution of the exit time of Ornstein-Uhlenbeck processes (\cite{APP05,GY03,Li04,RS88}), the Laplace transform and exponential form of the exit time (\cite{BS02,Di07}),
and the mean exit time (\cite{AF02,GJ08,ISR97}). In \cite{KM99}, the authors got some variational principles for the mean exit time moments of symmetric diffusions.

However, due to the lack of tools to deal with non-self adjoint operators, most prior results and applications of the exit time
 cited above, are restricted to symmetric diffusions.
Recently, Huang and Mao \cite{HM17,HM18} gave the variational principles for the hitting times of non-symmetric Markov chains  and from it some estimates and comparison theorems of the hitting times are obtained.
Motivated by the previous works, the goal of this paper is establishing some new kinds of variational principles for the exit time of non-symmetric diffusions and some comparison results.

In this paper, we consider the differential operator in divergence form:
\be\label{L}
L=\nabla \cdot a\nabla+b\cdot\nabla,
\de
where $a=(a_{ij})_{1\leq i,j\leq d}$ is a strictly positive definite matrix and $b=(b_1,\cdots,b_d)$ is a vector with $a_{ij},b_i\in C^1(\R^d)$.
Corresponding to the operator $L$,  there exists a unique solution $\P_x,x\in\overline\R^d$  satisfying the strong Markov property and the Feller property (see e.g. \cite[Theorem 1.11.1]{Pi95}), and we denote $X=(X_t)_{t\geq0}$ for the  diffusion {associated to} $L$.  {When $b\neq 0$,  $L$ is a non-self adjoint operator with respect to the Lebesgue measure and $X$ is a non-symmetric diffusion with respect to the Lebesgue measure. However, it is important to note that $L$ with $b\neq 0$ may be symmetric with respect to another measure. For instance, if $b$ is in the form of $b=a\nabla Q$ for some smooth function $Q$, then $L$ is self-adjoint with respect to the measure $\text{e}^{-Q(x)}{\rm d} x$.}

{For any open set $D\subset \R^d$, let $C^2_0(\bar{D})$ denotes the space of functions on $\bar{D}$ with compact support, and all of whose partial derivatives up to order 2 are continuous on $D$. For fixed} $\beta\geq0$, define a bilinear form  $\e_\beta^D$ associated with $L$ on ${C_0^2(\bar{D})}$ as
\be\label{df_beta}
\e_\beta(f,g):=	\e_\beta^{D}(f,g)=\int_D f(x)\big((\beta-L) g\big)(x){\rm d} x,\quad \text{for }f,g\in {C_0^2(\bar{D})}.
\de
{For simplicity, we will write $\e_0$ as $\e$.}
Denote by
$$\tau_D:=\inf\{t\ge 0:X_t\notin D\}$$
the first exit time from $D$ for $X$, and {let $\lambda_0(D)$ be the generalized principal eigenvalue} of $L$ on $D$. Then it is well known that
$$
\lambda_0(D)=\lim_{t\rightarrow\infty}\frac{1}{t}\log\sup_{x\in D}\mathbb{P}_x(\tau_D>t),
$$
see Lemma \ref{spec} below for more details. We use the notation $D'\Subset D$ if $D'$ is bounded and  $\bar{D'}\subset D$.
In the following, assume that $L$ satisfies the following {\bf Assumption A} on $D$:
\begin{itemize}
	\item[(A1)] $a_{ij},\ b_j\in C^{1,\alpha}(\bar{D})$;

	\item[(A2)] $v\cdot a(x)v>0$, for all $x\in\bar{D}$ and $v\in\R^d\backslash\{0\}$;

\item[(A3)]{ $\text{div}(b)=0$. Here $\text{div}(b):=\sum_{i=1}^d \partial b_i/\partial x_i$ is the divergence of $b$.}

\end{itemize}
Under the {\bf Assumption A},
we present some variational formulas of the exit time from domains which is our first main result.
{To do this,} let
$$
\cN_{\delta}:=\left\{h\in C^2_0(\bar{D}): {h|_{\partial D}=0\ \text{and }}\int_D h {\rm d} x=\delta\right\}\qquad \text{for }\delta=0, 1.
$$
\begin{thm}\label{vf_ht_unb}
Let $D\subset \R^d$ be a domain.
Assume that $L$ is the operator defined in \eqref{L} satisfying {\bf Assumption A} on any subdomain $D'\Subset D$ and {$\lambda_0( D)<0$}.
 Then for any $\beta>0$,
 \be\label{vf_htunb}
		\frac{\beta}{\int_D \big(1-\E_x[\exp(-\beta\tau_D)]\big){\rm d} x}=
		\inf_{f\in\cN_{1}}
		\sup_{g\in \cN_{0}}
		\e_\beta(f-g,f+g),
		\de
and
\begin{equation}\label{vf-meanunb}
		\frac{1}{\int_D\E_x[\tau_D] {\rm d} x}=
				\inf_{f\in\cN_{1}}
		\sup_{g\in \cN_{0}}
		\e(f-g,f+g).
\end{equation}
In particular, if $L$ is self-adjoint with respect to the Lebesgue measure, i.e. $b=0$, then \eqref{vf_htunb} and \eqref{vf-meanunb} {can be} reduced to
$$
         \frac{\beta}{\int_D \big(1-\E_x[\exp(-\beta\tau_D)]\big){\rm d} x}=\inf_{f\in\cN_{1}}\e_\beta(f,f)\quad \text{and}\quad
         \frac{1}{\int_D\E_x[\tau_D] {\rm d} x}=\inf_{f\in\cN_{1}}\e(f,f).
         $$
\end{thm}

In particular, we give the variational formula of the exponential moments of the exit time for symmetric diffusion $X$ corresponding to $L=\nabla \cdot a\nabla$ as follows:

\begin{thm}\label{expm1}
Let $D\subset \R^d$ be a domain.
Assume that $L$ is the operator defined in \eqref{L} with $b=0$ satisfying {\bf Assumption A} on any subdomain $D'\Subset D$ and $\lambda_0( D)<0$.
Then for any $\beta>0$ we have
\be
\frac{\beta}{\int_D\big(\E_x[\exp(\beta\tau_D)]-1\big) {\rm d} x}=\left(\inf_{f\in{\cN_1}}\e_{-\beta}(f,f)\right)\vee 0,\nn
\de
{where we use a notation $a\vee b:=\max\{a, b\}$.}
\end{thm}

\medskip
\begin{rem}
\begin{itemize}
\item[(1)] In \cite{KM99}, the variational principles for the mean $k$-th moments of the exit time of symmetric diffusions are obtained. In Theorem \ref{vf_ht_unb}, we generalize the variational principle for the mean exit time to non-symmetric diffusions. {As far as} we know, the variational principles for the Laplace transform of the exit time are new.

\item[(2)] In the context of Markov chains, \cite{HM17,HM18} obtain some variational principles for the hitting time of non-reversible ergodic Markov chains. Similar results of ergodic diffusions can be found in Section \ref{erg} below. Here we do not require the ergodicity of diffusion $X$.
\end{itemize}
\end{rem}

\medskip

We organize the rest of the paper as follows.
In Section \ref{apct}, we present {some} applications of our main result, Theorem \ref{vf_ht_unb}, with respect to the  comparison theorem and monotonicity law of the exit time.
Section \ref{mr} is devoted to the proof of Theorems \ref{vf_ht_unb} {and \ref{expm1}}  by introducing a general variational principle related to the Poisson's equations (see, \eqref{poi} below). We first derive the variational formulas for the exit time from bounded domains, and then we complete the proof of Theorems \ref{vf_ht_unb} {and \ref{expm1}}  using the approximation argument.
In Section \ref{erg}, we present extensions of our main results to ergodic diffusions.

\medskip

\noindent
{\bf Notation:}
For any $k\in \N_0$ and domain $D\subset \R^d$,
let $C^k(D)$ be the space of $k$ times continuously differentiable real-valued functions on $D$.
For any $\alpha\in(0, 1)$,   $C^{k,\alpha}(D)\subset C^k(D)$  that $k$th-order partial derivatives of $f\in C^{k,\alpha}(D)$ are locally H\"older continuous with exponent $\alpha$ in $D$.
For convenience, set $C^\alpha(D)=C^{0,\alpha}(D)$.
We denote by $C^k_0(D)$ the set of functions in $C^k(D)$ with compact support in $D$.
We use the notations $\text{vol}(A)=\int_A 1 \ {\rm d} x$ for  $A\subset \mathbb R^d$, $||v||$, $v\in \R^d$, for the Euclidean norm, and $B(n)\subset\R^d$ for the open ball with the radius $n$ centered at the origin.


\section{Applications}\label{apct}

We address the comparison theorems for the exit time of the diffusions associated with the drift $b\in \R^d$ of \eqref{L} in Section \ref{ct} as well as  associated with the coefficient matrices $a=(a_{ij})_{1\leq i,j\leq d}$ of \eqref{L} in Section \ref{ml}.


\subsection{Comparison theorems }\label{ct}

Let $a(x)=(a_{ij}(x))_{1\leq i,j\leq d},\ x\in\R^d$ be $d\times d$ strictly positive-definite matrices and $b(x)=(b_1(x),\cdots,b_d(x)),x\in\R^d$ be vectors on $\R^d$ with $a_{ij},b_i\in C^1(\R^d)$ and $\text{div}(b)=0$. 
 Denote
\begin{equation}\label{Lk}
L_{\gamma}=\nabla\cdot a\nabla+\gamma b\cdot\nabla,\quad \gamma\in\R.
\end{equation}
Then there exist diffusions $X^{(\gamma)}$ with infinitesimal generator $L_{\gamma}$ for $\gamma\in\R$.
Note that $X^{(0)}$ is symmetric diffusion since $L_{0}=\nabla\cdot a\nabla$ is self-adjoint with respect to the Lebesgue measure, while $X^{(\gamma)},\ \gamma\neq0$ are non-symmetric.
For any domain $D\subset \R^d$,
{denote by} $\tau_D^{(\gamma)}$ the exit time of diffusion $X^{(\gamma)}$ from $D$ and  $\lambda_0^{(\gamma)}(D)$ by the principal eigenvalue  associated to $L_{\gamma}$ which will be defined in Lemma \ref{spec} below.
By Theorem \ref{vf_ht_unb}, we obtain the comparison theorem between diffusions $X^{(\gamma)}$ according to $\gamma\in \R$.

\begin{thm}\label{comp_xk}
	Let $D\subset\R^d$ be a domain.
	Consider the operator $L_{\gamma}$ defined in \eqref{Lk} 
{satisfying {\bf Assumption A}} on any subdomain $D'\Subset D$, and $\lambda_0^{(\gamma)}(D)<0$.
Then
\begin{itemize}
\item[(1)] for any $\beta>0$,
$$
\int_D \big(1-\E_x[\exp(-\beta\tau_D^{(\gamma)})]\big){\rm d} x=\int_D\big( 1-\E_x[\exp(-\beta \tau_D^{(-\gamma)})]\big){\rm d} x.
$$
In particular, $\int_D\E_x [\tau_D^{(\gamma)} ]{\rm d} x= \int_D\E_x  [\tau_D^{(-\gamma)}] {\rm d} x$.

\item[(2)] For fixed $\beta>0$,
$\int_D \big(1- \E_x[\exp(-\beta\tau_D^{(\gamma)})]\big){\rm d} x$ and $\int_D\E_x [\tau_D^{(\gamma)}] {\rm d} x$ are non-increasing for $\gamma\in[0,\infty)$.
\end{itemize}
\end{thm}

\begin{proof}
(1) By \eqref{vf_htunb} in Theorem \ref{vf_ht_unb}, for any $\gamma\in\R$ and $\beta>0$ we obtain that
\begin{align}\label{vf_Lk}
&\quad\frac{\beta}{\int_D \big(1-\E_x[\exp(-\beta \tau_D^{(\gamma)})]\big){\rm d} x}\nn\\
&=\inf_{f\in\cN_{1}}\sup_{g\in \cN_{0}}\int_D\big(f(x)-g(x)\big)\big((\beta-L_{\gamma})(f+g)\big)(x){\rm d} x\nn\\
&=\inf_{f\in\cN_{1}}\sup_{g\in \cN_{0}}\left\{\int_D f(x)(\beta f-\nabla\cdot a\nabla f)(x){\rm d} x-\int_D g(x)(\beta g-\nabla\cdot a\nabla g)(x){\rm d} x\right.\\
&\quad \left.+2\gamma\int_D g(x) (b\cdot\nabla f)(x) {\rm d} x\right\}\nn\\
&=\inf_{f\in\cN_{1}}\sup_{g\in \cN_{0}}\left\{\int_D f(x)(\beta f-\nabla\cdot a\nabla f)(x){\rm d} x-\int_D g(x)(\beta g-\nabla\cdot a\nabla g)(x){\rm d} x\right.\nn\\
&\quad \left.+(-2\gamma)\int_D g(x) (b\cdot\nabla f)(x) {\rm d} x\right\}\nn\\
&=\frac{\beta}{\int_D \big(1-\E_x[\exp(-\beta \tau_D^{(-\gamma)})]\big){\rm d} x},\nn
\end{align}
{where we used the fact that $\int_D f(x)(b\cdot\nabla f)(x){\rm d} x=\int_D g(x)(b\cdot\nabla g)(x){\rm d} x=0$  by (A3) in the second equality},  and replaced $g$ by $-g$ in the third equality. This gives us the first assertion of (1). Similarly, the second assertion of (1) comes from \eqref{vf-meanunb}.

(2) For fixed $f\in\cN_{1}$, denote
$$\cN_{0,f}=\left\{g\in\cN_{0}:\int_Dg(x)(b\cdot\nabla f)(x){\rm d} x\ge 0\right\}.$$
We claim that the supremum in \eqref{vf_Lk} is equal to
$$
\sup_{g\in \cN_{0,f}}\left\{\int_D f(x)(\beta f-\nabla\cdot a\nabla f)(x){\rm d} x-\int_D g(x)(\beta g-\nabla\cdot a\nabla g)(x){\rm d} x+2\gamma\int_D g(x) (b\cdot\nabla f)(x) {\rm d} x\right\}.
$$
Therefore, from it we can obtain the first assertion in (2) immediately. To see that, consider $g\in\cN_{0}$ with $\int_Dg(x)(b\cdot\nabla f)(x){\rm d} x< 0$ and $\gamma\geq0$. Then $-g\in\cN_{0,f}$ and
$$
\aligned
&\quad\int_D f(x)(\beta f-\nabla\cdot a\nabla f)(x){\rm d} x-\int_D g(x)(\beta g-\nabla\cdot a\nabla g)(x){\rm d} x +2\gamma\int_D g(x) (b\cdot\nabla f)(x){\rm d} x\\
&<\int_D f(x)(\beta f-\nabla\cdot a\nabla f)(x){\rm d} x-\int_D (-g)(x)\big(\beta (-g)-\nabla\cdot a\nabla (-g)\big)(x){\rm d} x \\
&\quad +2\gamma\int_D (-g)(x) (b\cdot\nabla f)(x) {\rm d} x,
\endaligned
$$
so we validate our claim.
By a similar argument we obtain the second assertion in (2).
\end{proof}

\begin{rem}
Diffusions with growing drift have been studied, such as the behaviors of the spectral gap and the asymptotic variance under growing drift (see e.g. \cite{DLP16,FHPS10,HHS93,HNW15}). As we presented in Theorem \ref{comp_xk}, the behaviors of the Laplace transform of the exit time and the mean exit time have similar monotonicity properties.
\end{rem}


\subsection{Monotonicity law}\label{ml}

In this subsection, we consider two operators:
\be\label{L12}
L_{a_1}=\nabla\cdot(a_1\nabla)\quad \text{and}\quad
L_{a_2, b}=\nabla\cdot(a_2\nabla)+b\cdot\nabla,
\de
where $a_k(x)$, $k=1,2$, $x\in\R^d$, are $d\times d$ strictly positive-definite matrices and $b(x)$, $x\in\R^d$, are vectors on $\R^d$. Similar to Section \ref{ct}, under mild conditions, there exist diffusions $X^{(a_1)}$ and $X^{(a_2,b)}$ associated to $L_{a_1}$ and $L_{a_2,b}$, respectively.
For any domain $D\subset \R^d$, we denote by $\tau_{D}^{(a_1)}$ and $\tau_{D}^{(a_2,b)}$  the exit times for diffusions $X^{(a_1)}$ and $X^{(a_2,b)}$ from $D$, and
 denote by $\lambda_0^{(a_1)}(D)$ and $\lambda_0^{(a_2, b)}(D)$ the principal eigenvalues related to $L_{a_1}$ and $L_{a_2, b}$, respectively.
Using Theorem \ref{vf_ht_unb}, we obtain the following comparison theorem between diffusions $X^{(a_1)}$ and $X^{(a_2,b)}$.

\begin{thm}\label{T:mon}
Let $D\subset\R^d$ be a domain.
Consider the operator $L_{a_1}$ and $L_{a_2,b}$ defined as in \eqref{L12} satisfying {\bf Assumption A} on any subdomain $D'\Subset D$, and $\lambda_0^{(a_1)}(D), \lambda_0^{(a_2, b)}(D)<0$.
Additionally, we assume that 
$a_1\leq a_2$ in the sense that $a_2(x)-a_1(x)$ is non-negative definite for all $x\in\R^d$. Then for any $\beta>0$ we have
$$
\int_D \big(1- \E_x[\exp(-\beta\tau_{D}^{(a_2,b)})]\big){\rm d} x\leq \int_D  \big(1- \E_x[\exp(-\beta\tau_{D}^{(a_1)})]\big){\rm d} x
$$
and $\int_D\E_x[\tau_{D}^{(a_2,b)}] {\rm d} x\leq \int_D\E_x[\tau_{D}^{(a_1)}] {\rm d} x$. If in addition $b=0$, then
\begin{equation}\label{cr-exp}
\int_D\big(\E_x[\exp(\beta\tau_D^{(a_2,0)})]-1\big){\rm d} x\leq \int_D\big(\E_x[\exp(\beta\tau_D^{(a_1)})]-1\big) {\rm d} x.
\end{equation}
\end{thm}
\begin{proof}
Fix $\beta>0$. By \eqref{vf_htunb}, we obtain that
$$
\aligned
\frac{\beta}{\int_D\big(1- \E_x[\exp(-\beta\tau_{D}^{(a_2,b)})]\big){\rm d} x}&=\inf_{f\in\cN_{1}}\sup_{g\in \cN_{0}}\int_D \big(f(x)-g(x)\big)\big((\beta-L_{a_2,b})(f+g)\big)(x){\rm d} x\\
&\geq\inf_{f\in\cN_{1}}\int_D f(x)\big((\beta-L_{a_2,b}) f\big)(x){\rm d} x\\
&=\inf_{f\in\cN_{1}}\left\{\beta\int_Df^2(x)dx+\int_D \nabla f(x)\cdot a_2(x)\nabla f(x){\rm d} x\right\}\\
&\geq\inf_{f\in\cN_{1}}\left\{\beta\int_Df^2(x)dx+\int_D \nabla f(x)\cdot a_1(x)\nabla f(x){\rm d} x\right\}\\
&=\frac{\beta}{\int_D \left(1-\E_x[\exp(-\beta\tau_{D}^{(a_1)})]\right){\rm d} x}.
\endaligned
$$
Similarly, by \eqref{vf-meanunb} we also get $\int_D\E_x[\tau_{D}^{(a_2,b)}] {\rm d} x\leq \int_D\E_x[\tau_{D}^{(a_1)}] {\rm d} x$.

If in addition $b=0$, then both of $L_{a_1},L_{a_2, 0}$ are self-adjoint with respect to the Lebesgue measure. Hence, applying similar argument as above to Theorem \ref{expm1}, we obtain \eqref{cr-exp}.
\end{proof}

As a direct result of Theorem \ref{T:mon}, we obtain the monotonicity of variational formula for the exit time of diffusions $X^{(\varepsilon a_1)}$, associated to $L_{\varepsilon a_1}$, according to $\varepsilon>0$.

\begin{cor}\label{c:mon}
Let $D\subset \R^d$ be a domain. For any $\varepsilon>0$, we assume that the operator $L_{\varepsilon a_1}$,  defined in \eqref{L12} by replacing $a_1$ with $\varepsilon a_1$, satisfies {(A1) and (A2) in }{\bf Assumption A} on any subdomain $D'\Subset D$ and $\lambda_0^{(\varepsilon a_1)}(D)<0$.
 Then for any $\beta>0$,
	$$\int_D \big(1- \E_x[\exp(-\beta\tau_D^{(\varepsilon a_1)})]\big){\rm d} x,\quad \int_D\E_x[\tau_D^{(\varepsilon a_1)}]{\rm d} x\ \text{and }\
	\int_D\big(\E_x[\exp(\beta\tau_D^{(\varepsilon a_1)})]-1\big) {\rm d} x$$
	are non-increasing for $\varepsilon\in(0,\infty)$.
\end{cor}


\section{Proof of Theorems \ref{vf_ht_unb} and \ref{expm1}}\label{mr}

To prove the main result, we need some preparations.
Fix {a} domain $D\subset\R^d$. Let $L$ be the operator defined in \eqref{L} and let $X=(X_t)_{t\geq0}$ be the corresponding diffusion.
For $k\in \N_0$ and $\alpha\in(0,1)$, a domain $D\subset\R^d$ is said to have a $C^{k,\alpha}$-boundary if for each $x_0\in\partial D$, there is a ball $B$ centered at $x_0$ and a one-to-one mapping $\psi\in C^{k,\alpha}(B)$ from $B$ onto a set $A\subset\R^d$ such that
$$
\psi(B\cap D)\subset\{x\in \R^d:x_d>0\},\,\,
\psi(B\cap \partial D)\subset\{x\in\R^d:x_d=0\}
\,\,\text{and}\,\,
\psi^{-1}\in C^{k,\alpha}(A).
$$
Let
$\mathcal{B}_\alpha=\{f\in C^\alpha(\bar{D}):f=0\ \text{on}\ \partial D\}$ be the function space with a norm $||f||_{\alpha;D}:=\sup_{D}|f|+\sup_{x,y\in D,x\neq y}\frac{|f(x)-f(y)|}{|x-y|^\alpha}$, and
$$
\mathcal{D}_\alpha=\{f\in C^{2,\alpha}(\bar{D})\cap\mathcal{B}_\alpha:Lf=0\ \text{on}\ \partial D\}.
$$
{It can be checked that} the formal adjoint $(\widetilde{L},\widetilde{\mathcal{D}}_\alpha)$  of $(L,\mathcal{D}_\alpha)$ is given by 
$$
\widetilde{L}f=\nabla\cdot(a\nabla f)-b\cdot\nabla f-\text{div}(b)f\quad\text{for } f\in\widetilde{\mathcal{D}}_\alpha.
$$

\medskip

{We now} have the following properties for the spectrum of $(L,\mathcal{D}_\alpha)$ and $(\widetilde{L},\widetilde{\mathcal{D}}_\alpha)$ by \cite[Proposition 3.5.4 and Theorem 3.6.1]{Pi95}.
\begin{lem}\label{spec}
Let $D\subset\R^d$ be a bounded domain  with {$C^{2,\alpha}$-boundary} for some $\alpha\in(0,1)$.
Assume that $L$ is the operator defined in \eqref{L} satisfying {\bf Assumption A} on $D$.
Then
\begin{itemize}

\item[\rm(1)] the spectrum of $(L,\mathcal{D}_\alpha)$, denoted by $\sigma((L,\mathcal{D}_\alpha))$, consists only of eigenvalues;

\item[\rm(2)]
the principal eigenvalue  $$\lambda_0(D)=\sup\text{Re}\big(\sigma((L,\mathcal{D}_\alpha))\big)=\lim_{t\rightarrow\infty}\frac{1}{t}\log\sup_{x\in D}\mathbb{P}_x (\tau_D>t);$$

\item[\rm(3)]
$\lambda_0(D)=\sup\text{Re}\big(\sigma((\widetilde{L},\widetilde{\mathcal{D}}_\alpha))\big)$.
\end{itemize}

\end{lem}

\medskip

For $\beta\geq0$ and $\xi\in C^\alpha(\bar{D})$, we consider a pair of Poisson's equations as follows:

\begin{equation}\label{poi}
\begin{cases}
(\beta-L) u=\xi,&\text{in } D;\\
u=0, &\text{on }\partial D,
\end{cases} \quad  \text{and}\quad
\begin{cases}
(\beta-\widetilde{L}) u=\xi,&\text{in } D;\\
u=0, &\text{on }\partial D.
\end{cases}
\end{equation}
Then the next lemma which comes from \cite[Theorem 3.3.1]{Pi95} guarantees the {existence} of the solutions for the above Poisson's equations.

\begin{lem}\label{dp_slt}
Let $D\subset\R^d$ be a bounded domain  with {$C^{2,\alpha}$-boundary} for some $\alpha\in(0,1)$.
	Assume that $L$ is the operator defined in \eqref{L} satisfying {\bf Assumption A} on $D$. Then for any $\beta\geq0$ and $\xi\in C^\alpha(\bar{D})$, there exist unique solutions $u_\beta,\ \widetilde{u}_\beta\in C^{2,\alpha}(\bar{D})$ to Poisson's equations \eqref{poi}.
\end{lem}

\medskip
For $\beta>0$ (resp. $\beta=0$), let $\xi\equiv\beta$ (resp. $\xi\equiv1$) in Lemma \ref{dp_slt}.
Then the solutions of \eqref{poi} are
\begin{align*}
 u_\beta(x)&=1-\E_x[\exp(-\beta\tau_D)]\  (\text  {resp. } u_0(x)=\E_x[\tau_D]) \quad \text{and }\\
 \widetilde{u}_\beta(x)&=1-\E_x[\exp(-\beta\widetilde{\tau}_D)]\ (\text {resp. } \widetilde{u}_0(x)=\E_x[\widetilde{\tau}_D])\quad  \text{respectively, \quad  for $x\in\R^d$.}
\end{align*}
Here we use the notation $\widetilde{\tau}_\cdot$ for the exit time of diffusion associated to $\widetilde{L}$.


\subsection{The exit time on smooth bounded domains}\label{bd}

In this subsection, we  consider the case that $D\subset \R^d$ is a bounded domain with $C^{2,\alpha}$-boundary for some $\alpha\in(0,1)$. Let
$$
 \cN_\delta^\xi=\left\{h\in C^2(\bar{D}): h|_{\partial D}=0\ \text{and }\int_D h\,\xi {\rm d} x=\delta\right\}\qquad \text{for }\delta=0, 1.
$$
With the help of these function spaces, we obtain a variational formula for Poisson's equations {\eqref{poi} as follows.

\begin{prop}\label{main_g}
Let $D\subset \R^d$ be a bounded domain with $C^{2,\alpha}$-boundary for some $\alpha\in(0,1)$. Assume that $L$ is the operator defined in \eqref{L} satisfying {\bf Assumption A} on $D$ {and $\lambda_0( D)<0$}. { For any $\beta\geq0$ and non-zero $\xi\in C^\alpha(\bar{D})$}, denote by $u_\beta$, $\widetilde{u}_\beta$ the solutions of Poisson's equations \eqref{poi} respectively.
	Then
	\be\label{df_xi}
	1/\e_\beta(\widetilde{u}_\beta,u_\beta)
	=\inf_{f\in {\cN_1^\xi}}
	\sup_{g\in {\cN_0^\xi}}
	\e_\beta(f-g,f+g).
	\de
In particular, if $L$ is self-adjoint with respect to the Lebesgue measure, i.e., $b=0$, then \eqref{df_xi} {can be} reduced to
\be\label{df_xi_r}
1/\e_\beta(\widetilde{u}_\beta,u_\beta)=\inf_{f\in  \cN_1^\xi}\e_\beta(f,f).
\de
\end{prop}

\begin{proof}
Fixed $\beta\geq0$. { From Lemma \ref{dp_slt}, we can see that the solutions $0\neq u_\beta,\widetilde{u}_\beta\in C^{2,\alpha}(\bar{D})$ of Poisson's equations \eqref{poi} are well defined.}
Since $\widetilde{L}$ is the adjoint operator of $L$ with respect to the Lebesgue measure, we have
$$
\int_D \widetilde{u}_\beta(\beta-L) u_\beta {\rm d} x=\int_D u_\beta(\beta-\widetilde{L})\widetilde{u}_\beta {\rm d} x {=\int_D u_\beta(\beta-L)u_\beta {\rm d} x},
$$
which implies that
\be\label{ip}
{\e_\beta(u_\beta,u_\beta)=}\e_\beta(\widetilde{u}_\beta,u_\beta)=\int_D \widetilde{u}_\beta \xi {\rm d} x=\int_D u_\beta\xi {\rm d} x.
\de
{It is worth emphasizing that from $\beta\geq0>\lambda_0(D)$ and Assumption (A3),  it can be verified that each term in \eqref{ip} is positive.}

We now set $w_\beta=u_\beta\big(\int_D u_\beta\xi {\rm d} x\big)^{-1}$ and $\widetilde{w}_\beta=\widetilde{u}_\beta\big(\int_D \widetilde{u}_\beta\xi {\rm d} x\big)^{-1}$. Then it is easy to check that
\be
\bar{w}_\beta:=(w_\beta+\widetilde{w}_\beta)/2\in \cN_1^\xi \quad \text{and}\quad \hat{w}_\beta:=(w_\beta-\widetilde{w}_\beta)/2\in \cN_0^\xi.\nn
\de
For any $f\in\cN_1^\xi$, set $f_1=f-\bar{w}_\beta$. It can be checked that $f_1\in \cN_0^\xi$ since $\bar{w}_\beta\in \cN_1^\xi$. It follows this fact and \eqref{ip} that
\begin{equation}\label{geq2}
\e_\beta(\widetilde{w}_\beta,w_\beta)=1/\e_\beta(\widetilde{u}_\beta,u_\beta)\quad \text{and}\quad
\e_\beta(\widetilde{w}_\beta,f_1)=\frac{1}{\e_\beta(\widetilde{u}_\beta,u_\beta)}\int_D f_1\xi {\rm d} x=\e_\beta(f_1,w_\beta)=0.
\end{equation}
{In addition, using (A2) and (A3) again, we can see that 
$$
\e_\beta(f_1,f_1)=\beta \int_D f_1^2{\rm d}x+\e(f_1,f_1)\geq \e(f_1,f_1) =\int_D \nabla fa\cdot\nabla f{\rm d} x\geq 0.
$$
Combining this with \eqref{geq2},} we arrive at 
$$
\aligned
\e_\beta(f-\hat{w}_\beta,f+\hat{w}_\beta)&=\e_\beta(\widetilde{w}_\beta+f_1,w_\beta+f_1)\\	&=\e_\beta(\widetilde{w}_\beta,w_\beta)+\e_\beta(\widetilde{w}_\beta,f_1)+\e_\beta(f_1,w_\beta)+\e_\beta(f_1,f_1)\\
&\geq 1/\e_\beta(\widetilde{u}_\beta,u_\beta),
\endaligned
$$
which implies that
\be\label{leq}
	1/\e_\beta(\widetilde{u}_\beta,u_\beta)
	\leq\inf_{f\in \cN_1^\xi}\sup_{g\in\cN_0^\xi}\e_\beta(f-g,f+g).
	\de

Similarly, for any $g\in\cN_0^\xi$, let $g_1=g-\hat{w}_\beta$. Then $g_1\in \cN_0^\xi$ by the fact $\hat{w}_\beta\in \cN_0^\xi$.
Replacing $f_1$ by $g_1$ in \eqref{geq2}, the same argument as above yields that
\begin{align}\label{geq1}	\e_\beta(\bar{w}_\beta-g,\bar{w}_\beta+g)&=\e_\beta(\bar{w}_\beta-\hat{w}_\beta-g_1,\bar{w}_\beta+\hat{w}_\beta+g_1)\nn\\ &=\e_\beta(\widetilde{w}_\beta,w_\beta)+\e_\beta(\widetilde{w}_\beta,g_1)-\e_\beta(g_1,w_\beta)-\e_\beta(g_1,g_1)\nn\\
&\leq 1/\e_\beta(\widetilde{u}_\beta,u_\beta),
\end{align}
and this implies that
\be\label{geq}
	1/\e_\beta(\widetilde{u}_\beta,u_\beta)
	\geq\inf_{f\in \cN_1^\xi}\sup_{g\in\cN_0^\xi}\e_\beta(f-g,f+g).
	\de
Combining \eqref{leq} and \eqref{geq}, we obtain the first assertion.
In particular, if $b=0$, i.e., $L=\widetilde{L}$, for any $f\in\cN_1^\xi$ and $g\in \cN_0^\xi$, we have
$$
\e_\beta(f-g,f+g)=\e_\beta(f,f)-\e_\beta(g,g)\leq \e_\beta(f,f).
$$
Thus we get \eqref{df_xi_r} from \eqref{df_xi} immediately.
\end{proof}

\begin{rem}
There are some variational principles of non-symmetric diffusions in the literature. {In particular, \cite[Chapters 3 and 4]{Pi95}} provides a variational formula for {the principal eigenvalue} of non-symmetric diffusions. In \cite{LMS19}, the variational principle for the capacity of non-symmetric diffusions is obtained. Here we give variational formulas \eqref{df_xi} and \eqref{df_xi_r} for Poisson's equations \eqref{poi}.  Using that, we can give proofs of Theorems \ref{vf_ht_unb} {and \ref{expm1}} when $D$ is a bounded domain as below.
\end{rem}

\medskip\noindent
{\bf Proof of Theorem \ref{vf_ht_unb} for bounded domains}.
For any $\beta>0$, by \cite[Theorem 3.6.4]{Pi95} (see also \cite[Theorem 5.12]{Du15}),
$(1-\E_x[\exp(-\beta\tau_D)]:x\in \R^d)$ is the unique solution of Poisson's equation
$$
\begin{cases}
(\beta-L) u=\beta,& \text{in } D;\\
u=0,&\text{on }\partial D.
\end{cases}
$$
{Thus} $u_\beta:=((1-\E_x[\exp(-\beta\tau_D)])/\beta:x\in \R^d)$ is the unique solution of equation
\begin{equation}\label{hteq2}
\begin{cases}
(\beta-L) u=1,& \text{in } D;\\
u=0,&\text{on }\partial D.
\end{cases}
\end{equation}
Also \cite[Theorem 3.6.4]{Pi95} yields that
$(\E_x[\tau_D]:x\in\R^d)$ is the unique solution of Poisson's equation
\begin{equation}\label{hteq1}
\begin{cases}
-Lu=1,& \text{in } D;\\
u=0,&\text{on }\partial D,
\end{cases}
\end{equation}
Applying Proposition \ref{main_g} and \eqref{ip} to \eqref{hteq2} and \eqref{hteq1}, we obtain the desired results.\qed

\medspace

To prove Theorem \ref{expm1}, assume that $b=0$ in the rest of this subsection. For any $\beta>0$, define
\begin{align}\label{sol:poi_symm}
v_\beta(x)=\E_x[\exp(\beta\tau_D)], \quad x\in\R^d.
\end{align}
It is known that $v_\beta$ is a finite-valued function when $\beta<-\lambda_0(D)$. Furthermore, it is the unique solution to Poisson's equation
\begin{equation}\label{epn_poi}
\begin{cases}
(\beta+L)u=0,\quad&\text{in }D;\\
u=1,\quad&\text{on }\partial D.
\end{cases}
\end{equation}
If $\beta\geq-\lambda_0(D)$, then $v_\beta(x)=\infty$, $x\in D$(see e.g. \cite{Fr73,Kh59}).

\medspace

\noindent
{\bf Proof of Theorem \ref{expm1} for bounded domains}.
We first consider the case that $\beta<-\lambda_0(D)$.
Then $v_\beta$ defined in \eqref{sol:poi_symm} is the unique solution of equation \eqref{epn_poi}. So by a straight-forward computation we can see that $\bar{v}_\beta:=(v_\beta-1)/(\int_D v_\beta  {\rm d} x-\text{vol}(D))$ is the unique solution to
	$$
	\begin{cases}
	(\beta+L)u=-\beta/{\big(\int_D v_\beta  {\rm d} x-\text{vol}(D)\big)},\ &\text{in }D;\\
	u=0,\ &\text{on }\partial D.
	\end{cases}
	$$
Therefore, we have
\begin{equation}\label{epn_v}
	\e_{-\beta}(\bar{v}_\beta,\bar{v}_\beta)=\frac{\beta}{\int_D{v}_\beta {\rm d} x-\text{vol}(D)}
	=\frac{\beta}{\int_D\E_x[\exp(\beta\tau_D)]dx-\text{vol}(D)}.
\end{equation}
For any $f\in\cN_1$, let $f_1=f-\bar{v}_{\beta}$. It is obvious that $f_1 \in\cN_0$ by $\bar{v}_{\beta}\in \cN_1$. Combining this and the fact that $L$ is self-adjoint with respect to the Lebesgue measure,  $\e_{-\beta}(f_1,\bar{v}_\beta)=\e_{-\beta}(\bar{v}_\beta,f_1)=0$. Thus we have
\begin{equation}\label{eq:lower}
	\e_{-\beta}(f,f)=	\e_{-\beta}(f_1+\bar{v}_\beta,f_1+\bar{v}_\beta)=	\e_{-\beta}(f_1,f_1)+\e_{-\beta}(\bar{v}_\beta,\bar{v}_\beta)\geq \e_{-\beta}(\bar{v}_\beta,\bar{v}_\beta).
	\end{equation}	
In the last inequality above, we used the fact that $\e_{-\beta} (f_1,f_1)\geq0$ for $f_1\in\cN_0$ if $\beta<|\lambda_0(D)|$.
So we obtain our assertion by \eqref{epn_v} and \eqref{eq:lower} when  $\beta<|\lambda_0(D)|$.

\medskip
Assume now that  $\beta\geq |\lambda_0(D)|$. Then $v_\beta\equiv\infty$ by \cite[Lemma 1.1]{Fr73} and thus
\begin{equation}\label{beta>}
 \frac{\beta}{\int_Dv_\beta {\rm d} x-\text{vol}(D)}=0.
\end{equation}
Let $h$ be the solution to
$$
	\begin{cases}
	\big(-\lambda_0(D) +L\big)u=0,\ &\text{in }D;\\
	u=0,\ &\text{on }\partial D,
	\end{cases}
	$$
satisfying $\int_D h dx=1$. Here  the existence of $h$ is guaranteed by \cite[Theorem 3.5.5]{Pi95}.
Then $h\in \cN_1$ and
\begin{align*}
		\e_{-\beta}(h,h)=\big(-\beta-\lambda_0(D)\big)\int_D h^2{\rm d} x\leq 0 = \frac{\beta}{\int_Dv_\beta {\rm d} x-\text{vol}(D)},
\end{align*}
which gives the desired result.\qed

\subsection{The exit time on arbitrary domains}\label{unbd}

In this subsection, we turn to arbitrary domain $D\subset \R^d$ which is not necessarily bounded. We will use a approximation arguments to complete the proof of Theorems \ref{vf_ht_unb} and \ref{expm1}.

\medskip
\noindent
{\bf Proof of Theorem \ref{vf_ht_unb} for arbitrary domains}.
We only need to prove the assertions for the Laplace transform of the exit time since that for the mean exit time follows a similar path.
Let $\{D_n\}_{n=1}^\infty$ be a sequence of bounded domains with $C^{2,\alpha}$-boundaries such that $D_n\subset D_{n+1}$ and $\cup_{n=1}^\infty  D_n=D$.
Denote by $\tau_n=\inf\{t\geq0:X_t\notin D_n\}$ the exit time from $D_n$.
Let
$$\cN_\delta^{(n)}=\left\{h\in C^2(\bar{D}_n): h|_{\partial D_n}=0\ \text{and }\int_{D_n} h {\rm d} x=\delta\right\}\quad  \text{for}\  \delta=0, 1.
$$
For any $f_0\in \cN_{1}$, there is an integer $m>0$ large enough such that $\text{supp}(f_0)\subset D_m$.
 Since $\lambda_0(D_m){\leq\lambda_0(D)}<0$, from Section \ref{bd} we see that for any $\beta>0$,
$$
\aligned
\frac{\beta}{\int_{D_m}\big(1-\E_x[\exp(-\beta\tau_m)]\big){\rm d} x}&=
\inf_{f\in\cN_1^{(m)}}
\sup_{g\in \cN_0^{(m)}}
\e_\beta(f-g,f+g)\\
&\leq \sup_{g\in \cN_0^{(m)}}
\e_\beta(f_0-g,f_0+g)\leq \sup_{g\in\cN_{0}}
\e_\beta(f_0-g,f_0+g).
\endaligned
$$
Combining this with the fact $\tau_m\leq \tau_D$, we get
\be\label{leq_unb}
\frac{\beta}{\int_D \big(1-\E_x[\exp(-\beta\tau_D)]\big){\rm d} x}\leq \inf_{f\in\cN_{1}}
\sup_{g\in\cN_{0}}
\e_\beta(f-g,f+g).
\de
Let $u_{n,\beta},\ \widetilde{u}_{n,\beta}(n\geq 1)$ be the unique solutions of equations \eqref{hteq2} on $D_n$.
Define
$$
w_{n,\beta}=u_{n,\beta}\left(\int_{D_n} u_{n,\beta}{\rm d} x\right)^{-1},\quad \widetilde w_{n,\beta}=\widetilde{u}_{n,\beta}\left(\int_{D_n} \widetilde{u}_{n,\beta}{\rm d} x\right)^{-1}
$$
and
$$
\bar{w}_{n,\beta}=(w_{n,\beta}+\widetilde{w}_{n,\beta})/2,\quad \hat{w}_{n,\beta}=(w_{n,\beta}-\widetilde{w}_{n,\beta})/2.
$$
Since
\be\label{u_tildeu}
\int_{D_n} u_{n,\beta}{\rm d} x=\int_{D_n} u_{n,\beta}(\beta-\widetilde{L})\widetilde{u}_{n,\beta}{\rm d} x=\int_{D_n} (\beta-L)u_{n,\beta}\widetilde{u}_{n,\beta}{\rm d} x=\int_{D_n}\widetilde{u}_{n,\beta}{\rm d} x
\de
by the definitions of $u_{n,\beta}$ and $\widetilde{u}_{n,\beta}$, one has that $\bar{w}_{n,\beta}\in\cN_{1}$ and $\hat{w}_{n,\beta}\in\cN_{0}$.
We will prove that for any $n\geq 1$ and $g\in\cN_{0}$,
\begin{equation}\label{geq_unb1}
\e_\beta(\bar{w}_{n,\beta}-g,\bar{w}_{n,\beta}+g)\leq \frac{\beta}{\int_{D_n}\big(1-\E_x[\exp(-\beta\tau_n)]\big){\rm d} x}.
\end{equation}
Then by taking the supremum on the left-hand side of \eqref{geq_unb1} over $g\in \cN_{0}$,
 and letting $n\rightarrow \infty$, we can obtain
$$
\limsup_{n\rightarrow\infty}\sup_{g\in\cN_{0}}
\e_\beta(\bar{w}_{n,\beta}-g,\bar{w}_{n,\beta}+g)\le
\frac{\beta}{\int_{D}\big(1-\E_x[\exp(-\beta\tau_D)]\big){\rm d} x}.
$$
It implies that
\be\label{geq_unb}
\inf_{f\in\cN_{1}}
\sup_{g\in\cN_{0}}
\e_\beta(f-g,f+g)
\le \frac{\beta}{\int_{D}\big(1-\E_x[\exp(-\beta\tau_D)]\big){\rm d} x}
\de
from $\bar{w}_{n,\beta}\in\cN_{1}$.
Hence we obtain \eqref{vf_htunb} by combining \eqref{leq_unb} and \eqref{geq_unb}.

\medskip
It remains to show that  \eqref{geq_unb1} is in force. Indeed,
for any $g\in\cN_{0}$, let $g_{n,1}=g-\hat{w}_{n,\beta}$. Then $g_{n,1}\in \cN_{0}$ by $\hat{w}_{n,\beta}\in\cN_{0}$.
Arguing similarly as we did in \eqref{geq1},
\be\label{geq-unb1}
\aligned
\e_\beta(\bar{w}_{n,\beta}-g,& \bar{w}_{n,\beta}+g)=\e_\beta(\widetilde{w}_{n,\beta}-g_{n,1},w_{n,\beta}+g_{n,1})\\
&= \e_\beta(\widetilde{w}_{n,\beta},w_{n,\beta})-\e_\beta(g_{n,1},g_{n,1})-\e_\beta(g_{n,1},w_{n,\beta})+\e_\beta(\widetilde{w}_{n,\beta},g_{n,1}).
\endaligned
\de
For the last two terms, since $(\beta-L) w_{n,\beta}=(\beta-\widetilde{L}) \widetilde{w}_{n,\beta}=0$ in $\bar{D}_n^c$ and
$$
(\beta-L)w_{n,\beta}=(\beta-\widetilde{L})\widetilde{w}_{n,\beta}=\left(\int_{D_n} u_{n,\beta}{\rm d} x\right)^{-1}\quad \text{in }D_n
$$
by \eqref{hteq2} and \eqref{u_tildeu}, we get that
$$
\aligned
\e_\beta(\widetilde{w}_{n,\beta},g_{n,1})-\e_\beta(g_{n,1},w_{n,\beta})&=\int_D \widetilde{w}_{n,\beta} (\beta-L)  g_{n,1}{\rm d} x-\int_D g_{n,1} (\beta-L)  w_{n,\beta}{\rm d} x\\
&=\int_D (\beta-\widetilde{L})\widetilde{w}_{n,\beta}g_{n,1}{\rm d} x-\int_D g_{n,1} (\beta-L) w_{n,\beta}{\rm d} x\\
&=\left(\int_{D_n} u_{n,\beta}{\rm d} x\right)^{-1}\left(\int_{D_n} g_{n,1}{\rm d} x- \int_{D_n} g_{n,1}{\rm d} x\right)=0.
\endaligned
$$
Hence, putting this equality with \eqref{geq-unb1} and the fact that $ \e_\beta(g_{n,1},g_{n,1})\geq0$ yields that
$$
\e_\beta(\bar{w}_{n,\beta}-g,\bar{w}_{n,\beta}+g)\leq \e_\beta(\widetilde{w}_{n,\beta},w_{n,\beta}),
$$
which implies \eqref{geq_unb1} by the definitions of $w_{n,\beta}$ and $\widetilde{w}_{n,\beta}$.\qed

\bigskip

\noindent{\bf Proof of Theorem \ref{expm1} for arbitrary domains}.
For any domain $D\subset \R^d$, let $\{D_n\}_{n=1}^\infty$ be defined as in the proof of Theorem \ref{vf_ht_unb}, and let  $\tau_n:=\tau_{D_n}$  be the exit time from $D_n$.
Suppose first that $\beta<-\lambda_0(D)$.
Thanks to $\lambda_0(D_n)\leq\lambda_0(D)<0$, we have that $\beta<-\lambda_0(D_n)$ for any $n\geq1$.
So applying The proof of Theorem \ref{expm1} for bounded domain in Section \ref{bd} to $D_n$ and using the fact $\tau_n\leq \tau_D$ for $n\geq 1$, we have that
$$
\frac{\beta}{\int_D\big(\E_x[\exp(\beta\tau_D)]-1\big) {\rm d} x}\leq \frac{\beta}{\int_{D_n}\big(\E_x[\exp(\beta\tau_n)]-1\big){\rm d} x}
= \inf_{f \in\cN_{1}^{(n)}}\e_{-\beta}(f,f),
$$
where
$\cN_1^{(n)}:=\big\{h\in C^2(\bar{D}_n): \int_{D_n} h {\rm d} x=1,\ h=0\ \text{on}\ \partial D_n\big\}.$ It implies that
\be\label{expm_leq}
\frac{\beta}{\int_D\big(\E_x[\exp(\beta\tau_D)]-1\big){\rm d} x}\leq \inf_{f \in\cN_1}\e_{-\beta}(f,f).
\de

{In addition},  let $v_{n,\beta}=(\E_x[\exp(\beta\tau_n)], x\in\R^d)$ be the unique solutions of equations \eqref{epn_poi} corresponding to $D_n$. Then $\bar{v}_{n,\beta}:=v_{n,\beta}(\int_{D_n}v_{n,\beta}{\rm d} x)^{-1}\in \cN^{(n)}_{1}\subset {\cN_{1}}$. Moreover,
$$
\e_{-\beta}(\bar{v}_{n,\beta},\bar{v}_{n,\beta})=\frac{\beta}{\int_D\big(\E_x[\exp(\beta\tau_n)]-1\big) {\rm d} x}.
$$
Letting $n\rightarrow \infty$ in the above equality, by $\tau_n\rightarrow\tau_D$ as $n\rightarrow\infty$ we obtain
\be\label{expm_geq}
\lim_{n\rightarrow\infty}\e_{-\beta}(\bar{v}_{n,\beta},\bar{v}_{n,\beta})=\frac{\beta}{\int_D\big(\E_x[\exp(\beta\tau_D)]-1\big) {\rm d} x}.
\de
Hence, we obtain the desired result for $\beta<-\lambda_0(D)$ from \eqref{expm_leq} and \eqref{expm_geq}.

\medskip

Next, we turn to the case $\beta\geq-\lambda_0(D)$. We also have \eqref{beta>} from \cite[Lemma 1.1]{Fr73}.
By the definition of $\{D_n\}_{n=1}^{\infty}$, we can see that $\lambda_0(D_n)$ increases to  $\lambda_0(D)$ as $n\rightarrow\infty$. Therefore there exists $m>0$ such that $\beta>-\lambda_0(D_m)$.
Let $h$ be the solution to
$$
\begin{cases}
	\big(-\lambda_0(D_m)+L\big)u(x)=0,\ &\text{in }D_m;\\
	u=0,\ &\text{on }\partial D_m
\end{cases}
	$$
satisfying $\int_{D_m}h{\rm d} x=\int_D h{\rm d} x=1$.
since $D_m$ is bounded, the existence of $h$ is clear as the proof of Theorem \ref{expm1} for bounded domain in Section \ref{bd}. Thus we have $h\in{\cN_1}$ and
\begin{align*}
\e_{-\beta}(h, h)
= \big(-\beta-\lambda_0(D_m)\big)\int_{D} h^2{\rm d} x\leq0=\frac{\beta}{\int_D\big(\E_x[\exp(\beta\tau_D)]-1\big){\rm d} x}
\end{align*}
by \eqref{beta>},   which completes the proof.\qed

\section{Ergodic diffusions}\label{erg}
In this section, we consider the ergodic diffusions starting from its stationary distribution.
Let $a(x)=(a_{ij}(x))_{1\leq i,j\leq d},x\in\R^d$ be matrices such that $a_{ij}\in C^2(\R^d),\ 1\leq i,j\leq d$. Note that we do not require the symmetry of the matrix $a(x)$ {since we can cover any matrix $a=a(x)$ considering $\bar a=(a+a^T)/2$.}
We suppose that $a$ satisfies the following {\bf Assumption B}:
\begin{itemize}

\item[(B1)] There are constants $0<\Lambda_1<\Lambda_2$ such that
$$
\Lambda_1||v||^2\leq v\cdot a(x)v\leq\Lambda_2||v||^2,\quad \text{for all }x,v\in\R^d.
$$

\item[(B2)] There is a constant $C>0$ such that
$$
\sum_{i,j=1}^d a_{ij}(x)^2\leq C\qquad \text{for all }x\in \R^d.
$$
\end{itemize}
Let $V\in C^3(\R^d)$ satisfying that
\be\label{cond.3}
\int_{\R^d}\exp(-V(x)){\rm d} x<\infty\qquad\text{and }\qquad
\lim_{n\rightarrow\infty}\inf_{z\notin B(n)}V(z)=\infty,
\de
where $B(n)$ is a ball with radius $n$ centered at the origin.
Using $a$ and $V$, we define a differential operator
\be\label{gnrtn}
(\cL f)(x)=e^{V(x)}\nabla\cdot\big(e^{-V(x)}(a\nabla f)(x)\big), \qquad  f\in C^2_0(\R^d).
\de
We also assume that
there exist constants $r,c>0$ such that
\be\label{cond.4}
(\cL V)(x)\leq  -c\quad \text{for all $x$ with }||x||\geq r.
\de
By \cite[Theorem 1.13.1]{Pi95}, the regularities of $a$ and $V$ 
 imply that there exist a process $Y$ corresponding to $\cL$.
For
\begin{align*}
Z:=\int_{\R^d}\exp(-V(x)){\rm d} x<\infty,
\end{align*}
set the probability measure $\pi({\rm d} x)=\exp(-V(x)){\rm d} x/Z$ on $\R^d$.
Then by \cite[Theorem 6.1.3]{Pi95} (see also \cite[Section 2]{LMS19}), {\bf Assumption B},  \eqref{cond.3} and \eqref{cond.4} yield that process $Y$ is positive recurrent with stationary distribution $\pi$.
Denote by $\E_\pi$ the corresponding expectation starting from $\pi$.
For any $\beta\geq0$, define
$$
\e_{\pi,\beta}(f,g)=\int_{\R^d} f(x)\big((\beta-\cL) g\big)(x)\pi({\rm d} x),\quad f,g\in C^2_0(\R^d).
$$
Throughout this section, for fixed domain $D\subset \R^d$ we denote by $\tau_D$ and $\lambda_0(D)$ the exit time and the principal eigenvalue for $D$ of process $Y$ respectively, when no confusion is possible.
Denote
$$
\mathcal{N}_{\pi,\delta}=\left\{h\in C^2_0(\bar{D}): \pi(h)=\delta\right\},
\quad \delta=0, 1.
$$

Arguing similarly as we did in the proof of Theorem \ref{vf_ht_unb}, we obtain the following variational formulas for the exit time of process $Y$ starting from $\pi$.


\begin{prop}\label{P:erg1}
Assume that $\cL$ is the operator defined in \eqref{gnrtn} with {\bf Assumption B}, \eqref{cond.3} and \eqref{cond.4}. Let $D\subset\R^d$ be a domain with $\lambda_0(D)<0$.
Then for any $\beta>0$,
$$
\frac{\beta}{1-\E_\pi[\exp(-\beta\tau_D)]}=
		\inf_{f\in\cN_{\pi,1}}
		\sup_{g\in \cN_{\pi,2}}
		\e_{\pi,\beta}(f-g,f+g)
$$
and
$$\frac{1}{\E_\pi[\tau_D]}=\inf_{f\in\cN_{\pi,1}}
\sup_{g\in \cN_{\pi,2}}\mathscr{E}_\pi(f-g,f+g).$$
\end{prop}

In particular, if $ \cL$ is reversible(symmetric) with respect to $\pi$, i.e., $a$ is symmetric, similar to Theorem \ref{vf_ht_unb}, variational formulas in Proposition \ref{P:erg1} are reduced to following simple forms.

\begin{cor}\label{c:erg2}
Assume that $\cL$ is the operator defined in \eqref{gnrtn} with symmetric $a$,
{\bf Assumption B}, \eqref{cond.3} and \eqref{cond.4}. Let $D\subset\R^d$ be a domain with $\lambda_0(D)<0$.
Then for any $\beta>0$,
$$
\frac{\beta}{1-\E_\pi[\exp(-\beta\tau_D)]}=
		\inf_{f\in\cN_{\pi,1}}
		\e_{\pi,\beta}(f,f)\quad
\text{and}\quad \frac{1}{\E_\pi[\tau_D]}=\inf_{f\in\cN_{\pi,1}}
\e_\pi(f,f).
$$
\end{cor}



\medspace

\noindent
{\bf Acknowledgement:}\
L.-J. Huang acknowledges support from National Key R\&D Program of China No. 2023YFA1010400.
K.-Y. Kim acknowledges support from National Science and Technology Council of Taiwan (110-2115-M-005-009-MY3).
Y.-H. Mao acknowledges the support from National Key Research and Development Program of China (2020YFA0712901).


\bibliographystyle{plain}
\bibliography{vfht}

\end{document}